\documentclass[a4paper,11pt]{amsart}

\usepackage{amssymb, amsfonts, latexsym, amsthm, amsmath, verbatim}
\pdfoutput=1
\usepackage{array,longtable}
\usepackage{color}
\usepackage[all]{xy}
\usepackage{amscd}
\usepackage{mathrsfs} 
\usepackage[english]{babel}
\usepackage{tikz}
\usepackage{graphicx} % use it to trim pictures
\usepackage{standalone}
\usepackage{epstopdf}
\usepackage[paper=a4paper, text={155mm,218mm},centering]{geometry}
\definecolor{darkblue}{rgb}{0,0,0.4} 
\usepackage[colorlinks=true, citecolor=darkblue, filecolor=darkblue, linkcolor=darkblue,urlcolor=darkblue]{hyperref}
\usetikzlibrary{arrows,cd,decorations.pathmorphing,decorations.pathreplacing}
\usepackage{amscd}
\usepackage{enumitem}
\usepackage{transparent}
\usepackage{cancel}
\usepackage{tcolorbox}
\usepackage{makecell}
\usepackage{imakeidx}
\tikzstyle{crossing}=[circle,fill=white,minimum height=6pt,inner sep=0pt, outer sep=0pt, style={transform shape=false}]

\usepackage{caption} 
\usepackage{thm-restate}
\usepackage{cleveref}

\numberwithin{equation}{section}

\theoremstyle{plain}
\newtheorem{theorem}[equation]{Theorem}
\newtheorem{construction}[equation]{Construction}
\newtheorem{proposition}[equation]{Proposition}
\newtheorem{corollary}[equation]{Corollary}
\newtheorem{lemma}[equation]{Lemma}

\theoremstyle{definition}

\newtheorem{example}[equation]{Example}

\newtheorem{definition}[equation]{Definition}

\newtheorem{notation}[equation]{Notation}
\newtheorem*{ack}{Acknowledgments}

\numberwithin{figure}{section}

\usepackage{etoolbox}
\docsvlist{A,B,C,D,E,F,G,H,I,J,K,L,M,N,O,P,Q,R,S,T,U,V,W,X,Y,Z}
\docsvlist{N,Z,Q,C,R,F,A,W,M}
\docsvlist{A,B,C,D,E,F,G,H,I,J,K,L,M,N,O,P,Q,R,S,T,U,V,W,X,Y,Z,a,b,c,d,e,f,h,i,j,k,m,n,o,p,q,r,s,u,v,w,x,y,z}
 
\docsvlist{a,b,e,j,m,E,I,J,K}
\docsvlist{atan,CKR, codim,deg,Fix,grad,Hom,icodim, Id, Int,Kom,KR,lk, Mod,pcodim,Sym, supp}
\docsvlist{M,S}

\newcommand{\wt}[1]{\widetilde{#1}}
\newcommand{\wh}[1]{\widehat{#1}}

\def\rev{${}^\textrm{rev}$}
\def\prr{${}^\textrm{par}$}

\newcounter{nparcount}
\newcounter{sparcount}
\def\poscros{\raisebox{-0.1em}{\hbox to 0.8em{\includegraphics[width=0.8em,height=0.8em]{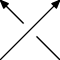}}}}
\def\negcros{\raisebox{-0.1em}{\hbox to 0.8em{\includegraphics[width=0.8em,height=0.8em]{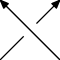}}}}
\def\moycros{\raisebox{-0.1em}{\hbox to 0.8em{\includegraphics[width=0.8em,height=0.8em]{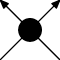}}}}
\def\rescros{\raisebox{-0.1em}{\hbox to 0.8em{\includegraphics[width=0.8em,height=0.8em]{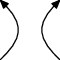}}}}
\def\vircros{\raisebox{-0.1em}{\hbox to 0.8em{\includegraphics[width=0.8em,height=0.8em]{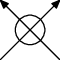}}}}
\def\arccros{\raisebox{-0.1em}{\hbox to 0.4em{\includegraphics[width=0.4em,height=0.8em]{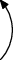}}}}
\def\Rvcros{\raisebox{-0.1em}{\hbox to 1.4em{\includegraphics[width=1.4em,height=0.8em]{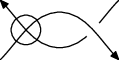}}}}
\def\Rcros{\raisebox{-0.1em}{\hbox to 1.4em{\includegraphics[width=1.4em,height=0.8em]{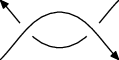}}}}

\def\bra#1{\langle #1\rangle}
\def\posbra{\bra{\poscros}}
\def\negbra{\bra{\negcros}}
\def\moybra{\bra{\moycros}}
\def\resbra{\bra{\rescros}}
\def\virbra{\bra{\vircros}}
\def\arcbra{\bra{\arccros}}
\def\Rvbra{\bra{\Rvcros}}
\def\Rbra{\bra{\Rcros}}

\title{Non-determinacy of HOMFLY-PT homology for diagrams}

\author{Maciej Borodzik}
\address{Institute of Mathematics\\ University of Warsaw
\\Warsaw, Poland}
\email{mcboro@mimuw.edu.pl}

\author{Mikhail Malashchuk}
\address{Institute of Mathematics\\ University of Warsaw
\\Warsaw, Poland}
\email{m.malashchuk@student.uw.edu.pl}

\begin{document}
\maketitle
  
\begin{abstract}
  We show that for any two knots $K_1$, $K_2$ in $\R^3$, there exist diagrams $D_1$ and $D_2$ that represent
  $K_1$ and $K_2$, respectively, such that
  the Khovanov--Rozansky triply graded homologies of $D_1$ and $D_2$ are isomorphic. The methods expand on Abel's paper \cite{Abel}.
\end{abstract}

\section{Introduction}
Let $K$ be a knot. Suppose $D$ is its oriented diagram. Khovanov and Rozansky in \cite{KR2} associate with $D$ a triply graded complex
$\CKR(D)$. The Reidemeister moves R1 and R2\prr{}, as well as the braid-like R3 move, induce graded isomorphisms of complexes. 
In \cite{KR2} it was suggested that the complex is not invariant under the R2\rev{} move.
%which was formally proved in \cite{Abel}. 
To set up the convention,
we show the R2\rev{} and R2\prr{} moves in Figure~\ref{fig:R2aR2b}.\footnote{In the literature, R2\prr{} is often referred to as the R2b move,
however the convention does not seem to be coherently adapted, therefore we suggest the notation R2\prr{}.}
\begin{figure}[h]
  \begin{tikzpicture}
    \node at(-3,1) {\includegraphics[width=2cm]{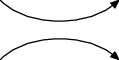}};
    \node at(-3,-1) {\includegraphics[width=2cm]{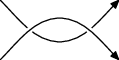}};

    \node at(3,1) {\includegraphics[width=2cm]{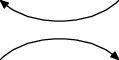}};
    \node at(3,-1) {\includegraphics[width=2cm]{pics/pictures-10.eps}};
    \draw[<->, thick] (-3,0.4) -- (-3,-0.4);
    \draw[<->, thick] (3,0.4) -- (3,-0.4);
  \end{tikzpicture}
  \caption{Left: the R2\prr{} (parallel) move, involving two strands going in the same direction. Right: the R2\rev{} (reverse) move,
  involving two strands going in opposite directions.}\label{fig:R2aR2b}
\end{figure}

In \cite{Abel}, Abel used an extension of Khovanov--Rozansky link homology to virtual links to show that there is a diagram $D$ of the unknot
whose Khovanov--Rozansky homology is that of the trefoil. He showed that the R2\rev{} move does not preserve the homology of $\CKR(D)$.
We expand on that result to prove the following, which is the main result of the article.
\begin{theorem}\label{thm:main}
  For any two knots $K_1$ and $K_2$, there exist diagrams $D_1$ and $D_2$ representing $K_1$ and $K_2$, respectively, such that
  $\CKR(D_1)$ is chain homotopy equivalent to $\CKR(D_2)$ (up to an overall grading shift).
\end{theorem}
Theorem~\ref{thm:main} exhibits a major obstacle to solving Kirby's problem \cite[Problem 1.34]{Kirby}. It is proved in Section~\ref{sec:main}, after necessary terminology is introduced in Section~\ref{sec:review}.

\begin{ack}
  The authors are grateful to L.H. Robert for fruitful conversations.
  MM would like to thank Professor Maciej Borodzik for his guidance. MB was supported by the NCN grant OPUS 2024/53/B/ST1/03470.
\end{ack}

\section{Review of HOMFLY-PT homology}\label{sec:review}
In order to prove Theorem~\ref{thm:main}, we set up some notation and review the definition of HOMFLY-PT homology. We rely on 
\cite{KR2,KR_virtual} and \cite{Rasmussen}. The main computational tool is in Subsection~\ref{sub:main_rel}, where several results
of \cite{Abel} are recalled.
\subsection{MOY graphs}
In \cite{MOY}, Murakami, Ohtsuki, and Yamada introduced trivalent graphs to facilitate the calculus of representations
of the quantum groups $\mathfrak{sl}_N$. A part of this construction is necessary for the construction of diagrammatic HOMFLY-PT homology.
Originally, a MOY graph of maximal weight $N$ is a trivalent oriented planar graph, possibly with boundary, with weights associated with each edge. A MOY graph
is required to satisfy the \emph{flow condition} meaning that the sum of weights of incoming edges at any vertex is equal to the sum of weights
of outgoing edges at that vertex.

A resolution of a standard, uncolored link yields MOY graphs with weights $1$ and $2$. When restricting to such MOY graphs, it is convenient
to assume that all weight $2$ edges are collapsed to a single vertex. This leads to the following definition.
\begin{definition}
  A \emph{reduced MOY graph} is a four-valent oriented planar graph, possibly with boundary, such that at each vertex there are two incoming
  and two outgoing edges. Moreover, the incoming edges are not separated by an outgoing edge.
\end{definition}
Throughout the paper, we will call a reduced MOY graph simply a \emph{MOY-graph}.

A marking of a MOY-graph is a choice of finitely many points on the edges of the graph such that the complement of these points
is a disjoint union of so-called \emph{elementary graphs}, as in Figure~\ref{fig:elementary}. By convention,
the marking points are not on the boundary of the graph.
\begin{figure}
  \begin{tikzpicture}
    \node at (2,0) {\includegraphics[height=1.5cm]{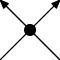}};
    \node at (-2,0) {\includegraphics[height=1.5cm]{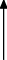}};
\end{tikzpicture}
\caption{Elementary graphs used in the definition of a marked MOY-graph.}\label{fig:elementary}
\end{figure}
\subsection{Chain complex associated with a MOY graph}\label{sub:chain}
Suppose $\Gamma$ is a marked MOY graph. We assign formal variables to the marks and the boundary points of $\Gamma$. Usually, to each marking
point, we associate a variable $t_i$ (so the number of these variables is equal to the number of markings). The boundary points of $\Gamma$
are divided into two types: incoming and outgoing. We associate variables $x_i$ to the incoming points and variables $y_j$ to the outgoing
boundary points. These variables are graded; by convention we assign the grading $2$ to $x_i,y_j,t_k$. We call this grading the \emph{quantum grading} or the \emph{$q$-grading}. We use the following notation:
\[\deg_q x_i=\deg_q y_j=\deg_q t_k=q^2,\]
where $q$ is a formal variable.

Given a marked graph $\Gamma$, we consider
\begin{itemize}
  \item the \emph{total ring}, $E^t(\Gamma)=\Q[\xx,\yy,\mathbf{t}]$;
  \item the \emph{edge ring}, $E(\Gamma)=\Q[\xx,\yy]$;
  \item the \emph{incoming} and \emph{outgoing rings}, $E^i(\Gamma)=\Q[\xx]$ and $E^o(\Gamma)=\Q[\yy]$, respectively.
\end{itemize}

Our goal is to associate a chain complex to any graph $\Gamma$. To achieve this, we introduce the following piece of notation. Let $R$ be a ring
with a $q$-grading.
With a homogeneous ring element $p$, we associate the complex $C_p$ given by
\[0\to a q^{\deg_q p} R\xrightarrow{p} R\to 0.\]
The homological grading of $C_p$ is called the \emph{Hochschild grading}. We use a formal variable $a$
to track this grading. Specifically, $a^k M$ is the module $M$ in the homological grading $k$.
Let $\cP=\{p_1,\dots,p_s\}$ be a finite collection of elements in $R$. We define the complex $C_{\cP}$ as the tensor product
of complexes $C_{p_i}$ for $i=1,\dots,s$.
\begin{definition}\label{def:elementary}
  The \emph{elementary} complex associated with an elementary MOY graph (as in Figure~\ref{fig:elementary}) is the complex
  \[C\arcbra=C_{y-x},\ \ C\moybra=C_{\cP},\]
  where $\cP=\{y_1+y_2-x_1-x_2,(y_1-x_1)(y_1-x_2)\}$. Here, the variables $x,y$, respectively $x_1,x_2,y_1,y_2$, are assigned
  as in Figure~\ref{fig:marking}.
\end{definition}
\begin{figure}
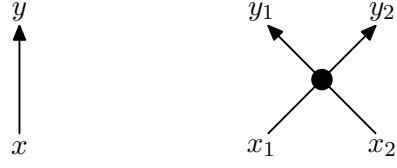

  \begin{tikzpicture}
    \node at (2,0) {\includegraphics[height=1.5cm]{pics/pictures-19.eps}};
    \node at (-2,0) {\includegraphics[height=1.5cm]{pics/pictures-20.eps}};
    \node[scale = 1] at (-2,0.9) {$y$};
    \node[scale = 1] at (-2,-0.9) {$x$};
    \node[scale = 1] at (1.2,-0.9) {$x_1$};
    \node[scale = 1] at (2.8,-0.9) {$x_2$};
    \node[scale = 1] at (1.2,0.9) {$y_1$};
    \node[scale = 1] at (2.8,0.9) {$y_2$};

\end{tikzpicture}
\caption{Variables assigned to the elementary graphs.}\label{fig:marking}
\end{figure}
Note that if $\Gamma$ is an elementary graph, the complex $C(\Gamma)$ is a complex of $E(\Gamma)$-modules.
To define the complex for an arbitrary graph, we introduce two operations: gluing and disjoint sum.
\begin{construction}\label{con:struction}
  Suppose $\Gamma$ is a disjoint union of two graphs $\Gamma_1,\Gamma_2$, and $C(\Gamma_1)$, $C(\Gamma_2)$ are complexes
  of $E(\Gamma_1)$-modules, respectively $E(\Gamma_2)$-modules. The complex $C(\Gamma_1\sqcup\Gamma_2)$ of $E(\Gamma)=E(\Gamma_1)\otimes E(\Gamma_2)$-modules is $C(\Gamma_1)\otimes_{\Q} C(\Gamma_2)$.

  If $\Gamma=\Gamma_1\cup\Gamma_2$ is a union of two oriented graphs along the set of markings $z_1,\dots,z_s$, where $z_1,\dots,z_s$
  can be incoming or outgoing variables, then $E(\Gamma)=E(\Gamma_1)\otimes E(\Gamma_2)/(z_1,\dots,z_s)$. We declare
  \[C(\Gamma)=C(\Gamma_1)\otimes_{\Q[z_1,\dots,z_s]} C(\Gamma_2).\]
\end{construction}
It is proved in \cite{KR2,Rasmussen} that Construction~\ref{con:struction}, together with Definition~\ref{def:elementary}, leads to a consistent
definition of $C(\Gamma)$ as a complex of $E(\Gamma)$-modules.
\begin{notation}
  The differential in the complex $C(\Gamma)$ is denoted $d_h$ (`h' standing either for `horizontal' or `Hochschild').
\end{notation}

\subsection{Bicomplex associated with a link diagram}
We begin by defining two maps between elementary complexes
\[\chi_i\colon C\resbra\to q^{-2}C\moybra,\ \chi_o\colon C\moybra\to C\resbra.\]
More specifically, with $E$ denoting the edge ring for both $\moycros$ and $\rescros$, we define the maps $\chi_i$ and $\chi_o$
as follows.
\[
  \begin{tikzcd}[ampersand replacement=\&, column sep = 2.5cm, row sep = 2cm]
    \rescros \ar[r,phantom,"="]\ar[d,"\chi_i"] 
  \& 
  a^2q^4E \ar[d,"1"]\ar[r,"\begin{pmatrix} y_1-x_1\\ y_2-x_2\end{pmatrix}"] 
  \& 
  aq^2E\oplus aq^2E \ar[r,"{\begin{pmatrix} x_2-y_2 & y_1-x_1 \end{pmatrix}}"] \ar[d,"{\begin{pmatrix} y_1-x_2 & 0 \\ 1 & 1\end{pmatrix}}"]
  \& 
  E\ar[d,"y_1-x_2"]
  \\
  \moycros \ar[r,phantom,"="] 
  \& 
  a^2q^6E \ar[r,"\begin{pmatrix} (y_1-x_1)(y_1-x_2)\\ y_1+y_2-x_1-x_2\end{pmatrix}"] 
  \&
  aq^2E\oplus aq^4E\ar[r,"A"]  %,"\begin{pmatrix} y_1-x_2 & 0 \\ 1 & 1 \end{pmatrix}"]
  \& 
  q^{2}E \\
    \moycros \ar[r,phantom,"="]\ar[d,"\chi_o"] 
  \& 
  a^2q^6E \ar[d,"y_1-x_2"]\ar[r,"\begin{pmatrix} (y_1-x_1)(y_1-x_2)\\ y_1+y_2-x_1-x_2\end{pmatrix}"] 
  \& 
  aq^2E\oplus aq^4E \ar[r,"A"] \ar[d,"{\begin{pmatrix} 1 & 0 \\ -1 & y_1-x_2\end{pmatrix}}"]
  \& 
  E\ar[d,"1"]
  \\
  \rescros \ar[r,phantom,"="] 
  \& 
  a^2q^4E \ar[r,"\begin{pmatrix} y_1-x_1 \\ y_2-x_2\end{pmatrix}"] 
  \&
  aq^2E\oplus aq^2E\ar[r,"{\begin{pmatrix}x_2-y_2 & y_1-x_1\end{pmatrix}}"]  %,"\begin{pmatrix} y_1-x_2 & 0 \\ 1 & 1 \end{pmatrix}"]
  \& 
  E,
\end{tikzcd}
\]
where $A=\begin{pmatrix} x_1+x_2-y_1-y_2 & (y_1-x_1)(y_2-x_2)\end{pmatrix}$.

It is a routine check that $\chi_i$ and $\chi_o$ commute with the differentials.
We now form two bicomplexes of $E$-modules corresponding to the crossings.
\[C\posbra = C\resbra\xrightarrow{\chi_i} tq^{-2}C\moybra,\ 
C\negbra = t^{-1}C\moybra\xrightarrow{\chi_o} C\resbra.\]
Here, $C\resbra$ and $C\moybra$ are complexes with the Hochschild differential. The maps $\chi_i$, respectively $\chi_o$, form the second
differential. We track the third grading by a formal grading variable $t$. 

We now pass to the definition of the bicomplex for a tangle diagram. We will use the formalism of a cube of resolutions. Take a link diagram $D$, and enumerate its crossings as $1,\dots,n$. For each
$c\in\{0,1\}^n$, we resolve the crossing as in Figure~\ref{fig:resolve}.
\begin{figure}
  \begin{tikzpicture}
    \node at (0,2) {\includegraphics[height=1cm]{pics/pictures-14.eps}};
    \node at (0,-2) {\includegraphics[height=1cm]{pics/pictures-15.eps}};
    \node at (4,0) {\includegraphics[height=1cm]{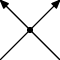}};
    \node at (-4,0) {\includegraphics[height=1cm]{pics/pictures-17.eps}};
    \draw[->,blue!50!black] (0.8,2) -- node [above, midway] {$\mathbf{1}$} node [below,midway] {$\cdot q^{-2}$} (3.8,0.8);
    \draw[->,red!50!black] (-0.8,2) -- node [above, midway] {$\mathbf{0}$} (-3.8,0.8);
    \draw[->,blue!50!black] (-0.8,-2) -- node [above, midway] {$\mathbf{1}$} (-3.8,-0.8);
    \draw[->,red!50!black] (0.8,-2) -- node [above, midway] {$\mathbf{0}$} (3.8,-0.8);
  \end{tikzpicture}
  \caption{Resolution of a diagram. The labels $0$ and $1$ denote the $0$-resolution, respectively the $1$-resolution, of the crossing. The
  $q^{-2}$ denotes the $q$-grading shift.}\label{fig:resolve}
\end{figure}
As a result, we obtain a MOY graph $D(c)$ for each element $c\in\{0,1\}^n$.  Each such element
comes with a quantum grading shift $q^{-2r_c}$, where $r_c$ is the number of occurrences of $1$ in the vector $c$ at coordinates
corresponding to positive crossings. All the graphs $D(c)$ have the same boundary, hence they share a common edge ring. Call it $E$.

Write $C_c$ for the complex $C(\Gamma)$ associated with $D(c)$. This is a complex of free bigraded $E$-modules. 
By shifting the grading by $q^{-2r_c}$ we keep track of the quantum grading.
The differential in $C_c$ 
is the Hochschild differential. If $c,c'$ differ by a single crossing, with $0$ appearing on $c$ and $1$ appearing on $c'$,
then $D(c)$ and $D(c')$ differ locally  by a change, either $\rescros\to q^{-2}\moycros$ or $\moycros\to\rescros$, depending on
the sign of the relevant crossing. We consider the map $d_{cc'}\colon C_c\to C_{c'}$ given by $\pm 1$ times $\chi_o$ or $\chi_i$
at the relevant crossing, tensored by the identity on the remaining part of the diagram. Here the sign $\pm 1$ comes from a so-called
\emph{sign assignment}, a standard technique in Khovanov homology.

Write 
\[\CKR(D)=\bigoplus_{c\in\{0,1\}^n}t^{|c|} C_c,\]
where $|c|$ is the number of $1$'s in $c$. The differentials $d_{cc'}$ are combined to obtain the differential in $\CKR(D)$. We call it
the \emph{vertical differential}.
\begin{definition}
  The bicomplex $\CKR(D)$ is called the \emph{Khovanov--Rozansky} complex associated with the diagram $D$.
\end{definition}

\subsection{Virtual MOY graphs and virtual complexes}
In \cite{KR_virtual} a construction of the Khovanov--Rozansky complex associated with a virtual MOY graph is given. To describe it, we begin by recalling that a \emph{virtual link diagram} $D$ is a diagram with an extra type of crossing, as in Figure~\ref{fig:virtual}.
\begin{figure}
  \includegraphics[width=1.5cm]{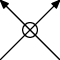}
  \caption{A virtual crossing.}\label{fig:virtual}
\end{figure}
There is an interpretation of a virtual link as a link drawn on a surface of positive genus. A virtual crossing is then regarded as a graphical presentation of a link whose strand going north-west to south-east is passing through a one-handle. 
Virtual links were introduced by Kauffman in \cite{Kaufman}.

A \emph{virtual MOY graph} is a planar graph with two types of vertices: vertices of the first type are MOY vertices, as in Figure~\ref{fig:elementary}; vertices of the second type are virtual
vertices, as in Figure~\ref{fig:virtual}. Put a marking on the graph and assign a variable to each marking point. With a marked graph $D$
we may associate a complex $C(D)$ as in Subsection~\ref{sub:chain}. The complex
for a MOY vertex is the same as in the non-virtual case. The complex for the virtual crossing with markings as in Figure~\ref{fig:virtual2} is given by
\[C\virbra:=C_{\{y_2-x_1,\,y_1-x_2\}}.\]
\begin{figure}
  \begin{tikzpicture}
    \node at (0,0) {\includegraphics[width=1cm]{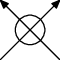}};
    \node at (0.7,0.7) {$y_2$};
    \node at (0.7,-0.7) {$x_2$};
    \node at (-0.7,-0.7) {$x_1$};
    \node at (-0.7,0.7) {$y_1$};
  \end{tikzpicture}
  \caption{Variables assigned to a virtual crossing.}\label{fig:virtual2}
\end{figure}
Note that the assignment is the same as if the $x_1$--$y_2$ and $x_2$--$y_1$ strands did not intersect.

Suppose now $D$ is a virtual link diagram. Let $n$ be the number of crossings, not counting virtual crossings. For $c\in\{0,1\}^n$,
we associate $D_c$, the corresponding resolution of $D$. The resolution is a virtual MOY graph, and we associate to it a complex $C(D_c)$. Finally,
we define the bicomplex $\CKR(D)$ analogously to the classical, non-virtual setting.

\subsection{Relations among the complexes $\CKR(D)$}\label{sub:main_rel}
We will now describe several relations among $\CKR(D)$ complexes, which will be used in the proof of Theorem~\ref{thm:main}.
Before we begin, we recall that $\CKR(D)$ is a bicomplex with two differentials: the Hochschild differential $d_h$ and
the vertical differential $d_v$. 

To phrase this more abstractly, suppose $R$ is a base ring. Assume $\cC$ is a bicomplex of graded $R$-modules (where the grading reflects the quantum
grading in $\CKR$), that is
$\cC=\bigoplus_{i,j} C_{ij}$, with $d_h\colon C_{ij}\to C_{i+1,j}$, $d_v\colon C_{ij}\to C_{i,j+1}$, and each of the $C_{ij}$ carries the quantum grading.

Let $\Kom(R)$ be the homotopy category of complexes of graded $R$-modules.
Define $C_j=(C_{\bullet,j},d_h)$ as an element in $\Kom(R)$. Then, $\cC$ is regarded as a complex in $\Kom(R)$.
\begin{definition}
  We say that two bicomplexes $\cC,\cC'$ are \emph{homotopy equivalent}, if they are homotopy equivalent as complexes of graded complexes
  in $\Kom(R)$.
\end{definition}
\begin{example}
  Suppose $\cC$ and $\cC'$ are two bicomplexes related by a Gaussian elimination, see \cite[Proposition 4.9]{Abel}. Then, they are
  homotopy equivalent.
\end{example}
Let $\cC$ be a bicomplex, $(C_{\bullet\bullet},d_h,d_v)$. By standard homological algebra, there exists a spectral sequence
abutting to the homology of $(C_{\bullet\bullet},d_h+d_v)$. Its $E_1$-page $E_{ij}^1$ is the homology of the complex $(C_{\bullet,j},d_h)$ in bidegree $(i,j)$. The map $d_v$ induces a differential $d_v^*\colon H(C_{ij},d_h)\to H(C_{i,j+1},d_h)$. The $E_2$-page is the homology
of the complex $(H(C_{ij},d_h),d_v^*)$.
\begin{lemma}
  Assume $\cC$ and $\cC'$ are homotopy equivalent. Then, the homotopy equivalence induces an isomorphism between $E_2$-pages
  of the spectral sequence.
\end{lemma}
\begin{proof}
  The proof is standard. Due to the subtlety that the isomorphism need not exist on the $E_1$-page, we give a quick sketch.

  Recall that $\cC$ and $\cC'$ are regarded as complexes $(C_j,d_v)$, $(C_j',d_v')$, where $C_j$ and $C_j'$ are complexes themselves, with differential $d_h$, respectively $d_h'$. Suppose $\cC$ and $\cC'$ are homotopy equivalent. That is, there exist maps $f\colon\cC\to\cC'$, $g\colon\cC'\to\cC$ and $H\colon\cC\to t^{-1}\cC$  with the property that
  $g\circ f -\Id= d_v H+Hd_v$, and $f,g,H$ commute with the differential $d_h$. In particular, $f$ and $g$ induce maps
  $f_*\colon H(C_j)\to H(C'_j)$, $g_*\colon H(C'_j)\to H(C_j)$, and $H$ induces $H_*\colon H(C_j)\to H(C_{j-1})$, satisfying $g_*f_*-\Id=d_v^*H_*+H_*d_v^*$,
  where $d_v^*\colon H(C_j)\to H(C_{j+1})$ is the induced differential. The latter formula implies that $g_*f_*$ induces an isomorphism
  on the homology $H(H(C_j),d_v^*)$, which is what we were supposed to show.
\end{proof}
\begin{example}
  It is usually not the case that homotopy equivalence induces an isomorphism on the $E_1$-pages. To see this, suppose $\cC$ is a bicomplex consisting of $C_{00},C_{01}$, which are free $R$-modules of rank $1$, with $d_h=0$ and $d_v\colon C_{00}\to C_{01}$ the identity. We write $\cC'$ for the zero complex. Then $\cC$ and $\cC'$ are homotopy equivalent; in fact, they are related by Gaussian elimination. But their $E_1$-pages clearly differ.
\end{example}

\begin{lemma}\label{lem:vir2b}
  The tangles in Figure~\ref{fig:vir2b} have homotopic bicomplexes.
  %The virtual Reidemeister 2b move as in Figure~\ref{fig:vir2b} induces a homotopyof the bicomplex $\CKR(D)$.
\end{lemma}
\begin{figure}
  \begin{tikzpicture}
    \node at (-3,0) {\includegraphics[width=2cm]{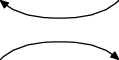}};
    \node at (3,0) {\includegraphics[width=2cm]{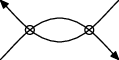}};
  \end{tikzpicture}
  \caption{Virtual Reidemeister 2\rev{} move.}\label{fig:vir2b}
\end{figure}
We remark that the tangles in Figure~\ref{fig:vir2b} have no non-virtual crossings, so their complexes have trivial vertical differential.
As a matter of fact, the isomorphism of the complexes of the tangles already holds at the level of Koszul complexes: the underlying modules are isomorphic and the isomorphism commutes with the differential $d_h$.

\begin{corollary}
  Suppose $D$ and $D'$ are two diagrams that differ by the move replacing two pieces of the diagrams as in Figure~\ref{fig:vir2b}. Then, $\CKR(D)$ and $\CKR(D')$ are homotopy equivalent.
\end{corollary}
Another result is proved in \cite[Proposition 4.2]{Abel}, see also \cite{KR_virtual}. In Figure~\ref{fig:move}, we present one of four variants of the move.
\begin{lemma}\label{lem:move}
  Suppose $D$ and $D'$ are virtual link diagrams differing as in Figure~\ref{fig:move}. Then, $\CKR(D)$ and $\CKR(D')$ are homotopy
  equivalent.
\end{lemma}
\begin{figure}
  \begin{tikzpicture}
    \node at (-3,0) {\includegraphics[width=2cm]{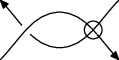}};
    \node at (3,0) {\includegraphics[width=2cm]{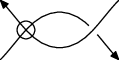}};
  \end{tikzpicture}
  \caption{One of the so-called $Z$-moves: swapping a virtual crossing with a regular crossing.}\label{fig:move}
\end{figure}

Next, we invoke the following result of Abel.
\begin{theorem}[see \expandafter{\cite[Theorem 4.11]{Abel}}]\label{thm:abel}
  The complexes $C\Rvbra$ and $tq^{-2}C\Rbra$ are homotopy equivalent.
\end{theorem}
From Theorem~\ref{thm:abel} it follows in particular that if two virtual diagrams differ by a change of $\Rvcros$ to $\Rcros$, then
they have homotopy equivalent chain complexes, up to an overall grading shift. From that result, we deduce the following, which is implicit in \cite{Abel}.
\begin{proposition}\label{prop:kain}
  If $D$ and $D'$ are two diagrams differing by a move as in Figure~\ref{fig:kain}, then $\CKR(D)$ and $t^2q^{-4}\CKR(D')$ 
  are homotopy equivalent.
\end{proposition}
\begin{figure}
  \begin{tikzpicture}
    \node at (-3,0) {\includegraphics[width=3cm]{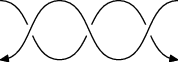}};
    \node at (3,0) {\includegraphics[width=3cm]{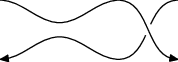}};
  \end{tikzpicture}
  \caption{The move inducing homotopy equivalence on $\CKR$ complexes.}\label{fig:kain}
\end{figure}

\begin{proof}
  The proof is essentially drawn in Figure~\ref{fig:kainproof}. Theorem~\ref{thm:abel} applied to the first two
  crossings allows us to pass from the first diagram of Figure~\ref{fig:kainproof} to the second. Next, we
  apply Theorem~\ref{thm:abel} to the middle and the right crossing. Note that the theorem is applied to the diagram rotated by $180$ degrees,
  so the virtual crossing appears on the right, not in the middle.

  We then apply a $Z$-move of Lemma~\ref{lem:move} to shift the virtual crossing to the right. Note that this move preserves the sign of the crossing,
  so the fourth diagram from the left has the crossing reversed as opposed to the previous diagrams. The last homotopy equivalence
  is obtained by applying Lemma~\ref{lem:vir2b}. Each application of Theorem~\ref{thm:abel} contributes a shift by $tq^{-2}$, which accounts
  for the total shift $t^2q^{-4}$.
\end{proof}
\begin{figure}
  \begin{tikzpicture}
    \node at (-6,0) {\includegraphics[width=2cm]{pics/pictures-2.eps}};
    \node at (-3,0) {\includegraphics[width=2cm]{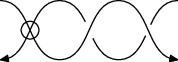}};
    \node at (0,0) {\includegraphics[width=2cm]{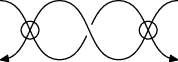}};
    \node at (3,0) {\includegraphics[width=2cm]{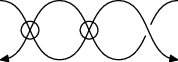}};
    \node at (6,0) {\includegraphics[width=2cm]{pics/pictures-4.eps}};
  \end{tikzpicture}
  \caption{Proof of Proposition~\ref{prop:kain}.}\label{fig:kainproof}
\end{figure}
%\begin{corollary}\label{cor:harmless}
%  Let $D_+$ be a virtual link diagram and suppose there is chosen a positive crossing on $D$. Let $D_-$ be the diagram with the crossing
%  changed to negative. There exists a diagram $D_+'$ representing the same link as $D_+$ such that
%  $\CKR(D_-)$ and $\CKR(D'_+)$ are homotopy equivalent.
%\end{corollary}
%\begin{proof}
%  Apply a Reidemeister 2b move next to the positive crossing on $D_+$ so as to obtain a diagram as in Figure~\ref{fig:harmless}.
%  Call the resulting link diagram $D'_+$. By Proposition~\ref{prop:kain}, the complexes $\CKR(D_-)$ and $\CKR(D'_+)$ are homotopy
%  equivalent.
%\end{proof}

\section{Proof of Theorem~\ref{thm:main}}\label{sec:main}
We begin with the following result, which is well known to the experts. For reference, see e.g.\ \cite[Section 3.1]{Adams}. % Adams' book Section 3.1
\begin{proposition}\label{prop:folklore}
  Suppose $D$ is a knot diagram. Choose a point $x$ on $D$, away from the crossings, and an orientation of $D$. Let $D'$ be the diagram obtained by changing the crossings
  of $D$ in such a way that, traversing $D'$ from $x$ in the chosen direction, every crossing is an overcrossing when we pass through it for the first time. Then $D'$ represents the unknot.
\end{proposition}
\begin{corollary}\label{cor:unknotting}
  Let $K_1$ and $K_2$ be two knots. Then, there exists a diagram $D_1$ representing $K_1$ such that crossing changes on $D_1$ result in a diagram $D_2$ representing $K_2$.
\end{corollary}
\begin{proof}
  Let $\wt{D}_1$ be a diagram of $K_1$ and $\wt{D}_2$ a diagram of $K_2$. Let $\wt{D}_1'$ and $\wt{D}_2'$ be the diagrams of the unknot
  obtained from $\wt{D}_1$ and $\wt{D}_2$ by Proposition~\ref{prop:folklore}. Then
  $D_1=\wt{D}_1\# \wt{D}_2'$ and $D_2=\wt{D}_1'\#\wt{D}_2$ represent $K_1$, respectively $K_2$, and they differ by crossing changes.
\end{proof}
\begin{figure}
  \begin{tikzpicture}
    \node at (-3,0) {\includegraphics[width=5cm]{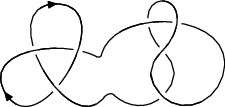}};
    \node at (3,0) {\includegraphics[width=5cm]{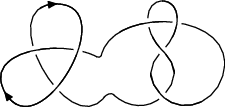}};
  \end{tikzpicture}
  \caption{The two diagrams represent a trefoil and a figure-eight knot. They differ by two crossing changes.}\label{fig:treight}
\end{figure}
\begin{example}
  The procedure of Corollary~\ref{cor:unknotting} is schematically presented in Figure~\ref{fig:treight}.
\end{example}

We are now ready to give a proof of Theorem~\ref{thm:main}.
\begin{figure}
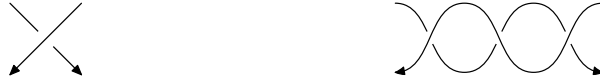

  \begin{tikzpicture}
    \node at (-3,0) {\includegraphics[height=1cm,angle=180]{pics/pictures-14.eps}};
    \node at (3,0) {\includegraphics[height=1cm]{pics/pictures-2.eps}};
  \end{tikzpicture}
  \caption{Proof of Theorem~\ref{thm:main}.}\label{fig:harmless}
\end{figure}
\begin{proof}[Proof of Theorem~\ref{thm:main}]
  Let $K_1$ and $K_2$ be two knots and let $D_1,D_2$ be the diagrams from Corollary~\ref{cor:unknotting}. Let $c_1,\dots,c_s$
  be the crossings on $D_1$ that need to be reversed in order to change $D_1$ to $D_2$. Let $\epsilon_i$, $i=1,\dots,s$
  be the choice of signs at these crossings: if $\epsilon_i=1$, then, while transforming $D_1$ into $D_2$, we change the crossing
  at $c_i$ from negative to positive; if $\epsilon_i=-1$, we change it from positive to negative.

  Let $\wh{D}_1$ and $\wh{D}_2$ be the diagrams obtained from $D_1$ and $D_2$ by the following procedure. Take $D_1$ and a crossing $c_i$
  such that $\epsilon_i=-1$. Modify the diagram near $c_i$ by adding a Reidemeister R2\rev{} move as in Figure~\ref{fig:harmless}. Do this for
  all crossings with $\epsilon_i=-1$, and let $\wh{D}_1$ be the resulting diagram. By construction, $\wh{D}_1$ represents $K_1$.

  Likewise, in the diagram $D_2$, we perform the same operation for all $c_i$ such that $\epsilon_i=1$. Let $\wh{D}_2$ be the resulting diagram. By construction, $\wh{D}_2$ represents $K_2$. Now, let $D$ be the diagram obtained from $D_1$ by changing the crossings $c_i$ with $\epsilon_i=-1$ from positive to negative, and leaving the crossings $c_i$ with $\epsilon_i=1$ unchanged. By Proposition~\ref{prop:kain} applied at each crossing $c_i$ with $\epsilon_i=-1$, the complexes $\CKR(\wh{D}_1)$ and $\CKR(D)$ are chain homotopy equivalent, up to an overall
  grading shift.

  By construction, $D$ is also obtained from $D_2$ by changing all the crossings $c_i$ with $\epsilon_i=1$ from negative to positive. Hence, applying Proposition~\ref{prop:kain} again,
  we obtain a chain homotopy equivalence between $\CKR(\wh{D}_2)$ and $\CKR(D)$, up to an overall grading shift. Combining the two equivalences, $\CKR(\wh{D}_1)$ and $\CKR(\wh{D}_2)$ are chain homotopy equivalent up to an overall grading shift, which proves the theorem.
  %with $D_1:=\wh{D}_1$ and $D_2:=\wh{D}_2$.
\end{proof}

\bibliographystyle{amsalpha}
\def\MR#1{}
\bibliography{bib}

\typeout{Number of still unresolved questions: \arabic{nparcount}.}
\typeout{Number of resolved questions: \arabic{sparcount}.}

\end{document}